\documentclass[10pt,a4paper]{article}
\usepackage{amssymb}
\usepackage{amsmath}

\usepackage{tikz}
\usetikzlibrary{matrix,arrows}

\newcommand{\qed}{%
  \ifmmode 
   \eqno{\qedsymbol}
  \else
    \leavevmode\unskip\penalty9999 \hbox{}\nobreak\hfill\hbox{\qedsymbol}
  \fi
}
\newcommand{\qedsymbol}{\leavevmode\vrule height 1.2ex width 1.1ex depth -.1ex}
\newenvironment{proof}{\begin{trivlist}\item[\hskip%
\labelsep{\bf Proof.\quad}]}%
{\hfill\qed\rm\end{trivlist}}

\newtheorem{theorem}{Theorem}[section]
\newtheorem{corollary}[theorem]{Corollary}
\newtheorem{proposition}[theorem]{Proposition}
\newtheorem{lemma}[theorem]{Lemma}

\mathchardef\emptyset="001F
\font\Bbb=msbm10 at 12 truept
\def\im{\hbox{\rm im}}
\def\dom{\hbox{\rm dom}}
\def\o{\overline}
\def\X{{\mathcal X}}
\def\Fr{{\mathcal F_{\mathcal R}}}
\def\Pr{{\mathcal P}}
\def\F{{\mathcal F}}
\def\E{{\mathcal E}}
\def\L{{\mathcal L}}
\def\R{{\mathcal R}}
\def\C{{\mathcal C}}
\def\D{{\mathcal D}}
\def\FP{\mathcal{FP}}
\def\IFP{\mathcal{IFP}}
\def\RIFP{\mathcal{RIFP}}
\def\PE{\mathcal{PE}}
\def\U{{\mathcal U}}
\def\Sp{{\mathcal E_{\mathcal S}}}
\def\UX{{\mathcal U_{\mathcal X}}}
\def\UPr{{\mathcal U_{\mathcal P}}}
\def\UFP{{\mathcal U_{\mathcal FP}}}
\def\RX{{\mathcal R_{\mathcal X}}}
\def\S{\hbox{\Bbb S}}
\def\cof{\hbox{\rm cof}}
\def\fib{\hbox{\rm fib}}
\def\ret{\hbox{\rm ret}}
\def\Fix{\hbox{\rm Fix}}
\def\tensor{\otimes}

\begin{document}

\title{Weak Factorization Systems for $S-$acts}
\author{Alex Bailey and James Renshaw\\\small Department of Mathematics\\
\small University of Southampton\\
\small Southampton, SO17 1BJ\\
\small England\\
\small Email: alex.bailey@soton.ac.uk\\
\small j.h.renshaw@maths.soton.ac.uk}
\date{January 2013}
\maketitle

\begin{center}
\end{center}

\begin{abstract}
The concept of a weak factorization system has been studied extensively in homotopy theory and has recently found an application in one of the proofs of the celebrated flat cover conjecture, categorical versions of which have been presented by a number of authors including Rosick\'y~\cite{rosicky-02}. One of the main aims of this paper is to draw attention to this interesting concept and to initiate a study of these systems in relation to flatness of $S-$acts and related concepts.
\end{abstract}



\section{Introduction}
Let $S$ be a monoid. Throughout, unless otherwise stated, all acts will be right $S-$acts. We refer the reader to~\cite{howie-95} for basic results and terminology in semigroups and monoids and to  \cite{ahsan-08} and \cite{kilp-00} for those concerning acts over monoids.
The object of this paper is in part to draw to the attention of the semigroup community the concept of a {\em weak factorization system}, to recasts in terms of acts over monoids some known results from category theory concerning weak factorization systems and to introduce some new results and examples that are connected with the concepts of covers of $S-$acts, flatness and related ideas.

\smallskip

After some introductory results and comments in section 1 we define the necessary concepts related to weak factorization systems and prove some standard results, many of which are already known to the category theory community. We give some examples of weak factorization systems of $S-$acts in section 3 which have connections with covers and flatness of $S-$acts. In section 4 we show that in some cases we can `generate' weak factorization systems from certain sets of mappings and finish by giving a pointer to a possible application relating to covers of centred $S-$acts.

\smallskip

Let $S$ be a monoid and let $\X$ be a class of $S-$acts. We shall have occasion to consider directed colimits of $S-$acts where the index set is regarded as an ordinal. For details of directed colimits in general see~\cite{ahsan-08}, \cite{bailey-12} or \cite{kilp-00}. We shall (informally) consider a {\em class} as a collection of {\em sets}, or viewed another way, a set is a class that is a member of another class. We shall consider an {\em ordinal} as a transitive set well-ordered by $\in$. In particular, if $\alpha$ is an ordinal then $\alpha=\{\beta: \beta\text{ is an ordinal and }\beta<\alpha\}$.
We refer the reader to~\cite{jeck-06} for a more formal treatment of axiomatic set theory and ordinals and a more detailed discussion on the differences between classes and sets.

Let $\lambda$ be an infinite ordinal. By a {\em$\lambda-$sequence} we mean a directed system of $S-$acts and $S-$maps $(A_\alpha,\phi_{\alpha,\beta}:A_\alpha\to A_\beta)_{\alpha\le\beta<\lambda}$ with a directed colimit $(A_\lambda,\phi_{\alpha})_{\alpha<\lambda}$ and such that for every limit ordinal $\gamma\le\lambda$, $(A_\gamma,\phi_{\alpha,\gamma})_{\alpha<\gamma}$ is the directed colimit of the directed system $(A_\alpha,\phi_{\alpha,\beta})_{\alpha\le\beta<\gamma}$.
Let $\C$ be a class of $S-$maps and let $(A_\alpha,\phi_{\alpha,\beta})$ be a $\lambda-$sequence in which for $\beta<\beta+1<\lambda$, $\phi_{\beta,\beta+1}\in\C$. Then the map $\phi_{0,\lambda}:A_0\to A_\lambda$ is called the {\em transfinite composition} of the maps $\{\phi_{\alpha,\beta}|\alpha\le\beta<\lambda\}$.

\medskip

Let $\gamma$ be a cardinal. An ordinal $\lambda$ is said to be {\em $\gamma-$filtered} if it is a limit ordinal and if $A\subseteq\lambda$ and $|A|\le\gamma$ then $\sup(A)<\lambda$. For example, if $\gamma$ is finite then $\omega$ is $\gamma-$filtered. As a generalisation, notice that for a given infinite cardinal $\gamma$ then $\gamma^+$, the successor cardinal to $\gamma$, is $\gamma-$filtered and so $\gamma-$filtered ordinals exist for each cardinal $\gamma$. To see that we require $\lambda=\gamma^+$ here rather than $\gamma$, consider the well-known example $\gamma = \aleph_\omega$ and let $A=\{\aleph_n | n<\omega\}$. Then $A\subseteq\gamma$ but $\sup(A)=\gamma$.


%

\smallskip

The following rather technical result, which will be useful later, is effectively a `transfinite' version of \cite[Lemma 2.3]{bailey-12}.
\begin{lemma}\label{small-lemma}
Let $S$ be a monoid and let $A$ be an $S$-act. Let $\gamma\ge\max\{|S|,|A|\}$ be an infinite cardinal and let $\lambda$ be a $\gamma-$filtered ordinal. Suppose that $(A_\alpha,\phi_{\alpha,\beta})$ is a $\lambda-$sequence of $S-$acts with a directed colimit $(A_\lambda,\phi_{\alpha})_{\alpha<\lambda}$ and let $f:A\to A_\lambda$ be an $S-$map. Then there exists $\delta<\lambda$ and an $S-$map $g:A\to A_\delta$ such that $f=\phi_\delta g$.
\end{lemma}
\begin{proof}
Note that by~\cite[Theorem 2.2]{bailey-12}, for each $a\in A$ there exists $\beta_a$ such that $f(a)\in \im(\phi_{\beta_a})$. Then $|\{\beta_a:a\in A\}|\le|A|\le\gamma$ and so $\alpha=\sup\{\beta_a:a\in A\}<\lambda$. Hence $\im(f)\subseteq\im(\phi_\alpha)$ and  $f$ will factor through a function $g:A\to A_\alpha$. Now again by~\cite[Theorem 2.2]{bailey-12} it follows that for each pair $(a,s)\in A\times S$ there exists $\alpha\le \beta_{a,s}<\lambda$ such that $\phi_{\alpha,\beta_{a,s}}(g(a)s) = \phi_{\alpha,\beta_{a,s}}(g(as))$. Let $\delta=\sup\{\beta_{a,s}|(a,s)\in A\times S\}$. Since $|\{\beta_{a,s}|(a,s)\in A\times S\}|\le\gamma$ then $\delta<\lambda$ and $f$ factors through an $S-$map $A\to A_\delta$ as required.
\end{proof}

\medskip

Let $f:X\to Y$ be an $S-$monomorphism and consider the right $S-$congruence $\tau$ on $Y$ given by
$
\tau=\im f\times\im f\cup 1_Y.
$
Let $Y/X=Y/\tau$ and denote the element $y\tau$ by $\o y$.

\begin{lemma}[{\cite[Lemma 2.6]{renshaw-86}}]\label{rees-composition-lemma}
Let $S$ be a monoid and let $f:X\to Y$ and $g:Y\to Z$ be $S-$monomorpisms. Then $\o g: Y/X\to Z/Y$ given by $\o g(\o y)= \o{g(y)}$ is an $S-$monomorphism and $Z/X\cong(Z/Y)/(Y/X)$.
\end{lemma}

Recall (\cite{renshaw-02}) that a map $f:A\to B$ is said to be {\em stable} if for every left $S-$map $g:X\to Y$, whenever $b\tensor g(x)=f(a)\tensor y$ in $B\tensor_SY$ then there exists $a'\in A, x'\in X$ such that $b\tensor g(x)=f(a')\tensor g(x')$. Recall also that an $S-$act $A$ is {\em flat} if for all left $S-$monomorphisms $g:X\to Y$ the induced map $A\tensor_S X\to A\tensor_S Y$ is one to one.

\begin{lemma}[{Cf.~\cite[Lemma 2.1]{renshaw-02}}]\label{rees-flat-lemma1}
Let $S$ be a monoid and let $f:X\to Y$ be an $S-$monomorphism. If $Y/X$ is flat then $f$ is stable. If $Y$ is flat and $f$ is stable then $Y/X$ is flat.
\end{lemma}

The following is easy to show
\begin{lemma}\label{flat-retract-lemma1}
Let $S$ be a monoid and let $X$ be a retract of a flat $S-$act $Y$. Then $X$ is flat.
\end{lemma}

\begin{lemma}[{\cite[Lemma 1.2]{renshaw-91}}]\label{rees-stable-lemma1}
Let $S$ be a monoid and let $f:A\to B$ be an $S-$monomorphism and let $g:B\to C$ be a stable monomorphism. Then $\o g:B/A\to C/A$ is a stable monomorphism.
\end{lemma}

It is well-known that if
$$
\begin{tikzpicture}[description/.style={fill=white,inner sep=2pt}]
\matrix (m) [matrix of math nodes, row sep=3em,
column sep=2.5em, text height=1.5ex, text depth=0.25ex]
{A&C\\ B& P\\ };
\path[->,font=\scriptsize]
(m-1-1) edge node[auto,left] {$f$} (m-2-1)
(m-1-2) edge node[auto,right] {$g$} (m-2-2)
(m-2-1) edge node[auto,below] {$v$} (m-2-2)
(m-1-1) edge node[auto,above] {$u$} (m-1-2);
\end{tikzpicture}
$$
is a pushout diagram then $P\cong(B\dot\cup C)/\rho$ where $\rho=\{(f(a),u(a)):a\in A\}^\sharp$ and $v(b)=b\rho, g(c)=c\rho$. Notice then that if $p\in P$ then either $p\in\im(v)$ or $p\in\im(g)$. Notice also that if $f$ is a monomorphism and if $v(b)=v(b')$ then either $b=b'$ or there exists $a,a'\in A$ such that $b=f(a), u(a)=u(a'), f(a')=b$. We shall make use of this fact later.

\medskip

Let $S$ be a monoid, let $A$ be an $S-$act and let $\X$ be a class of $S-$acts closed under isomorphisms. By an $\X$-{\em precover} of $A$ we mean an $S-$map $g: P\to A$ for some $P\in \X$ such that
for every $S-$map $g':P'\to A$, for $P'\in \X$, there exists an $S-$map $f:P'\to P$ with $g'=gf$.
$$
\begin{tikzpicture}[description/.style={fill=white,inner sep=2pt}]
\matrix (m) [matrix of math nodes, row sep=3em,
column sep=2.5em, text height=1.5ex, text depth=0.25ex]
{P & A\\ 
&P'\\};
\path[->,font=\scriptsize]
(m-2-2) edge node[auto,below left] {$f$} (m-1-1)
(m-2-2) edge node[auto, right] {$g'$} (m-1-2)
(m-1-1) edge[->] node[auto,above] {$g$} (m-1-2);
\end{tikzpicture}
$$

If in addition the precover satisfies the condition that each $S-$map $f:P\to P$ with $gf=g$ is an isomorphism, then we shall call it an {\em $\X-$cover}. We shall frequently identify the (pre)cover with its domain. For more details of covers and precovers of acts we refer the reader to \cite{bailey-12} and \cite{bailey-13}.

\section{Weak Factorization Systems}

Much has been written in recent years on weak factorization systems, mostly in more general categorical terms and some without explicit proof. As might be expected, notation and terminology used seem to vary widely. For background on the ideas contained in this section see for example ~\cite{beke-00},~\cite{hovey-88} and~\cite{rosicky-02}.

Let $\X$ be a class of $S-$acts closed under isomorphisms. Let $f:A\to B$ and $g:C\to D$ be $S-$maps such that given any commutative square of $S-$maps
$$
\begin{tikzpicture}[description/.style={fill=white,inner sep=2pt}]
\matrix (m) [matrix of math nodes, row sep=3em,
column sep=2.5em, text height=1.5ex, text depth=0.25ex]
{A&C\\ B& D\\ };

\path[->,font=\scriptsize]
(m-1-1) edge node[auto,left] {$f$} (m-2-1)
(m-1-2) edge node[auto,right] {$g$} (m-2-2)
(m-2-1) edge node[auto,below] {$v$} (m-2-2)
(m-1-1) edge node[auto,above] {$u$} (m-1-2);
\end{tikzpicture}
$$
there exists an $S-$map $h:B\to C$ such that $hf=u$ and $gh=v$. In this case we say that $g$ has the {\em right lifting property} with respect to $f$ and that $f$ has the {\em left lifting property} with respect to $g$. Let $\C$ be a class of $S-$maps and let
$$
\C^\Box=\{g | g\text{ has the right lifting property with respect to each }f\in\C\}
$$
$$
{}^\Box\C=\{f | f\text{ has the left lifting property with respect to each }g\in\C\}
$$
\begin{lemma}\label{box-algebra-lemma}
Let $S$ be a monoid and let $\C\subseteq\D$ be classes of $S-$maps. Then
\begin{enumerate}
\item $\C\subseteq\left({}^\Box\C\right)^\Box$ and $\C\subseteq{}^\Box\left(\C^\Box\right)$,
\item ${}^\Box\D\subseteq{}^\Box\C$ and $\D^\Box\subseteq\C^\Box$,
\item ${}^\Box\C={}^\Box\left(\left({}^\Box\C\right)^\Box\right)$ and $\C^\Box=\left({}^\Box\left(\C^\Box\right)\right)^\Box$.
\end{enumerate}
\end{lemma}
\begin{proof}
Let $\C\subseteq\D$ be classes of $S-$maps.
\begin{enumerate}
\item Let $g\in\C$ and $f\in{}^\Box\C$. Then $f$ has the left lifting property with respect to $g$ and so $g$ has the right lifting property with respect to $f$. Since this holds for all $f\in{}^\Box\C$ then $g\in\left({}^\Box\C\right)^\Box$.

From the definition of $\C^\Box$, $g$ has the left lifting property with respect to every element of $\C^\Box$ and so $g\in{}^\Box\left(\C^\Box\right)$.

\item Let $f\in{}^\Box\D$. Then $f$ has the left lifting property with respect to all maps in $\D$ and so in particular with respect to all maps in $\C$. Hence $f\in{}^\Box\C$.

Let $g\in\D^\Box$. Then $g$ has the right lifting property with respect to all maps in $\D$ and so in particular with respect to all maps in $\C$. Hence $g\in\C^\Box$.

\item ${}^\Box\C\subseteq{}^\Box\left(\left({}^\Box\C\right)^\Box\right)$ follows from part (1) while the reverse inclusion follows from parts (1) and (2). Similarly for the other equality.
\end{enumerate}
\end{proof}

Notice that if $f$ is any $S-$isomorphism then $f\in{}^\Box\C\cap\C^\Box$ for every class $\C$ of $S-$maps.

\medskip

Let $\C$ be a class of $S-$maps. The maps in $\text{\rm fib}(\C) = \left({}^\Box\C\right)^\Box$ are called the {\em fibrations} of $\C$ whilst the maps in $\cof(\C) ={}^\Box\left(\C^\Box\right)$ are called the {\em cofibrations} of $\C$ (see~\cite{hovey-88}). From Lemma~\ref{box-algebra-lemma} we see that $\C^\Box = \cof(\C)^\Box$ and ${}^\Box\C={}^\Box\text{\rm fib}(\C)$ and so $\cof(\cof(\C))=\cof(\C), \fib(\fib(\C))=\fib(\C)$. In addition, if $\C\subseteq\D$ then $\cof(\C)\subseteq\cof(\D)$.

If $f:A\to B, g:A\to C$ are $S-$maps such that there exist maps $\alpha:C\to B$ and $\beta:B\to C$ with $\beta\alpha=1_C, \alpha g=f$ and $\beta f=g$ then we say that {\em $g$ is a retract of $f$}. In categorical terms, $g$ is a retract of $f$ in the coslice-category $(A\downarrow\C)$.

\begin{lemma}\label{cof-lemma}
Let $S$ be a monoid and let $\C$ be any class of morphisms. Then 
\begin{enumerate}
\item pushouts of maps in $\C$ are in $\cof(\C)$,
\item $\cof(\C)$ is closed under transfinite compositions,
\item $\cof(\C)$ is closed under retracts.
\end{enumerate}
\end{lemma}

\begin{proof}\ 

\begin{enumerate}
\item Consider the following pushout diagram where $f\in\C$.
$$
\begin{tikzpicture}[description/.style={fill=white,inner sep=2pt}]
\matrix (m) [matrix of math nodes, row sep=3em,
column sep=2.5em, text height=1.5ex, text depth=0.25ex]
{A&X\\
 B& Y\\ };
\path[->,font=\scriptsize]
(m-1-1) edge node[auto,left] {$f$} (m-2-1)
(m-1-2) edge node[auto,right] {$\bar f$} (m-2-2)
(m-2-1) edge node[auto,below] {$\beta$} (m-2-2)
(m-1-1) edge node[auto,above] {$\alpha$} (m-1-2);
\end{tikzpicture}
$$
We want to show that $\bar f\in\cof(\C)$. Let $g:C\to D$ be in $\C^\Box$ and consider the commutative diagram
$$
\begin{tikzpicture}[description/.style={fill=white,inner sep=2pt}]
\matrix (m) [matrix of math nodes, row sep=3em,
column sep=2.5em, text height=1.5ex, text depth=0.25ex]
{A&X&C\\
 B& Y&D\\ };
\path[->,font=\scriptsize]
(m-1-2) edge node[auto,above] {$u$} (m-1-3)
(m-2-2) edge node[auto,below] {$v$} (m-2-3)
(m-1-3) edge node[auto,right] {$g$} (m-2-3)
(m-1-1) edge node[auto,left] {$f$} (m-2-1)
(m-1-2) edge node[auto,right] {$\bar f$} (m-2-2)
(m-2-1) edge node[auto,below] {$\beta$} (m-2-2)
(m-1-1) edge node[auto,above] {$\alpha$} (m-1-2);
\end{tikzpicture}
$$
Since $f\in\C$ then there exists $h:B\to C$ such that $hf=u\alpha$ and $gh=v\beta$. Since $Y$ is a pushout then there exists a unique $\bar h:Y\to C$ with $\bar h\bar f=u$ and $\bar h\beta=h$. So $g\bar h\beta=gh=v\beta$. In a similar way, there exists a unique map $\gamma:Y\to D$ such that $\gamma\bar f=gu$ and $\gamma\beta=v\beta$. However it is then clear that both $v$ and $g\bar h$ satisfy these equations and so the result follows.
\item Suppose that $f_1:A\to B$ and $f_2:B\to C$ are such that $f_1,f_2\in \cof(\C)$ and consider the commutative diagram
$$
\begin{tikzpicture}[description/.style={fill=white,inner sep=2pt}]
\matrix (m) [matrix of math nodes, row sep=3em,
column sep=2.5em, text height=1.5ex, text depth=0.25ex]
{A&X\\
 C& Y\\ };
\path[->,font=\scriptsize]
(m-1-1) edge node[auto,left] {$f_2f_1$} (m-2-1)
(m-1-2) edge node[auto,right] {$g$} (m-2-2)
(m-2-1) edge node[auto,below] {$v$} (m-2-2)
(m-1-1) edge node[auto,above] {$u$} (m-1-2);
\end{tikzpicture}
$$
where $g\in\C^\Box$. Then there exists $h_1 : B\to X$ and $h_2:C\to X$ such that
$$
\begin{tikzpicture}[description/.style={fill=white,inner sep=2pt}]
\matrix (m) [matrix of math nodes, row sep=3em,
column sep=2.5em, text height=1.5ex, text depth=0.25ex]
{A&X\\
B&\\
 C& Y\\ };
\path[->,font=\scriptsize]
(m-2-1) edge node[auto,above] {$h_1$} (m-1-2)
(m-3-1) edge node[auto,right] {$h_2$} (m-1-2)
(m-1-1) edge node[auto,left] {$f_1$} (m-2-1)
(m-2-1) edge node[auto,left] {$f_2$} (m-3-1)
(m-1-2) edge node[auto,right] {$g$} (m-3-2)
(m-3-1) edge node[auto,below] {$v$} (m-3-2)
(m-1-1) edge node[auto,above] {$u$} (m-1-2);
\end{tikzpicture}
$$
commutes. Hence $f_2f_1\in\cof(\C)$.

Now suppose, by way of transfinite induction, that $\lambda$ is a limit ordinal and that for all $i\le j<\lambda$, $f_{i,j}:A_i\to A_j$ are such that $f_{i,j}\in\cof(\C)$ and that if $j$ is a limit ordinal then $(A_j,f_{i,j})$ is a colimit of $(A_i,f_{i,k})_{i\le k<j}$. Let $(A_\lambda,f_{i,\lambda})$ be the directed colimit and consider the following commutative diagram
$$
\begin{tikzpicture}[description/.style={fill=white,inner sep=2pt}]
\matrix (m) [matrix of math nodes, row sep=3em,
column sep=2.5em, text height=1.5ex, text depth=0.25ex]
{A_0&X\\
 A_\lambda& Y\\ };
\path[->,font=\scriptsize]
(m-1-1) edge node[auto,left] {$f_{0,\lambda}$} (m-2-1)
(m-1-2) edge node[auto,right] {$g$} (m-2-2)
(m-2-1) edge node[auto,below] {$v$} (m-2-2)
(m-1-1) edge node[auto,above] {$u$} (m-1-2);
\end{tikzpicture}
$$
with $g\in\C^\Box$. For any $j\ge i\ge 0$ there exists $h_{0,i}:A_i\to X, h_{0,j}:A_j\to X$ such that
$$
\begin{tikzpicture}[description/.style={fill=white,inner sep=2pt}]
\matrix (m) [matrix of math nodes, row sep=3em,
column sep=2.5em, text height=1.5ex, text depth=0.25ex]
{A_0&X\\
A_i&\\
A_j&\\
A_\lambda&Y\\ };
\path[->,font=\scriptsize]
(m-2-1) edge node[auto,above] {$h_i$} (m-1-2)
(m-3-1) edge node[auto,right] {$h_j$} (m-1-2)
(m-1-1) edge node[auto,left] {$f_{0,i}$} (m-2-1)
(m-2-1) edge node[auto,left] {$f_{i,j}$} (m-3-1)
(m-3-1) edge node[auto,left] {$f_{j,\lambda}$} (m-4-1)
(m-1-2) edge node[auto,right] {$g$} (m-4-2)
(m-4-1) edge node[auto,below] {$v$} (m-4-2)
(m-1-1) edge node[auto,above] {$u$} (m-1-2);
\end{tikzpicture}
$$
commutes and so the colimit property means there exists a unique $h_\lambda:A_\lambda\to X$ such that $h_\lambda f_{j,\lambda} = h_j$ for all $j\le\lambda$. So we deduce that $h_\lambda f_{0,\lambda} = h_1f_{0,1}=u$ and the colimit property applied to $Y$ allows us to deduce that $gh_\lambda = v$.
\item Suppose that $f:A\to B\in\cof(\C)$ and $g:A\to C$ are such that there exist maps $\alpha:C\to B$ and $\beta:B\to C$ with $\beta\alpha=1_C, \alpha g=f$ and $\beta f=g$. Consider the following commutative diagram
$$
\begin{tikzpicture}[description/.style={fill=white,inner sep=2pt}]
\matrix (m) [matrix of math nodes, row sep=3em,
column sep=2.5em, text height=1.5ex, text depth=0.25ex]
{&A&X\\
 B& C&Y.\\ };
\path[->,font=\scriptsize]
(m-1-2) edge node[auto,above] {$u$} (m-1-3)
(m-2-2) edge node[auto,below] {$v$} (m-2-3)
(m-1-3) edge node[auto,right] {$h$} (m-2-3)
(m-1-2) edge node[auto,left] {$f$} (m-2-1)
(m-1-2) edge node[auto,right] {$g$} (m-2-2)
(m-2-1.340) edge[<-] node[auto,below] {$\alpha$} (m-2-2.200)
(m-2-1.20) edge node[auto,above] {$\beta$} (m-2-2.160);
\end{tikzpicture}
$$
Then there exists $\gamma:B\to X$ such that $\gamma f=u$ and $h\gamma=v\beta$. It is then easy to check that $\psi:C\to X$ given by $\psi = \gamma\alpha$ is such that $\psi g=u$ and $h\psi=v$ as required.
\end{enumerate}
\end{proof}

\medskip

Although pullbacks of $S-$acts do not always exists we have a partial dual of the previous result, the proof of which is omitted.

\begin{lemma}
Let $S$ be a monoid and let $\C$ be a class of $S-$maps. Then
\begin{enumerate}
\item pullbacks of epimorphisms in $\C$ are in $\fib(\C)$,
\item $\fib(\C)$ is closed under composites,
\item $\fib(\C)$ is closed under retracts.
\end{enumerate}
\end{lemma}

\bigskip

The following would appear to be well-known (see for example~\cite[Remark 2.4]{rosicky-02}) and the proof is almost identical to that for Lemma~\ref{cof-lemma}.

\begin{lemma}\label{box-c-lemma}
Let $S$ be a monoid and let $\C$ be any class of $S-$morphisms. Then ${}^\Box\C$ and $\C^\Box$ contain all isomorphisms and
\begin{enumerate}
\item pushouts of maps in ${}^\Box\C$ are in ${}^\Box\C$,
\item ${}^\Box\C$ is closed under transfinite compositions,
\item $^\Box\C$ is closed under retracts.
\end{enumerate}
\end{lemma}

Let $\C$ be a class of $S-$morphisms. Denote by $\ret(\C)$ the class of all $S-$morphisms that are retracts of transfinite compositions of pushouts of maps in $\C$. It follows immediately from Lemma~\ref{cof-lemma} that $\C\subseteq\ret(\C)\subseteq\cof(\C)$ and from Lemma~\ref{box-c-lemma} that $\ret({}^\Box\C)={}^\Box\C$. If $\C$ is a class of $S-$maps such that $\C=\ret(\C)$ then we shall say that $\C$ is {\em saturated}.
\bigskip

A {\em weak factorization system} $(\L,\R)$ consists of two classes $\L$ and $\R$ of morphisms satisfying
\begin{enumerate}
\item $\R=\L^\Box$ and $\L={}^\Box\R$,
\item any morphism $h$ has a factorization $h=gf$ with $f\in\L$ and $g\in\R$.
\end{enumerate}

Notice that the `lifting' map involved in property (1) need not be unique. If it is always unique then we refer to the system as a {\em factorization system}. The following result appears to be well-known~\cite[Remark III.4]{adamek-02} and provides a more practical way to check whether the system $(\L,\R)$ forms a weak factorization system.
\begin{proposition}\label{split-factorization-proposition}
Let $S$ be a monoid. Then $(\L,\R)$  is a weak factorization system if and only if
\begin{enumerate}
\item any morphism $h$ has a factorization $h=gf$ with $f\in\L$ and $g\in\R$,
\item for all $f\in\L, g\in\R, f$ has the left lifting property with respect to $g$,
\item if $f:A\to B$ and $f':X\to Y$ is such that there exist maps $\alpha:B\to Y$ and $\beta:A\to X$ then
\begin{enumerate}
\item if $\alpha f\in\L$ and if $\alpha$ is a split monomorphism then $f\in\L$,
\item if $f'\beta\in\R$ and if $\beta$ is a split epimorphism then $f'\in\R$.
\end{enumerate}
\end{enumerate}
\end{proposition}
\begin{proof}
Suppose that $(\L,\R)$ is a weak factorization system.
Then clearly properties (1) and (2) hold and so let $f,f', \alpha, \beta$ be as in the statement of property (3) and suppose that $\gamma:X\to A$ is the splitting map, i.e. $\gamma\beta=1_X$. Suppose that we have a commutative square
$$
\begin{tikzpicture}[description/.style={fill=white,inner sep=2pt}]
\matrix (m) [matrix of math nodes, row sep=3em,
column sep=2.5em, text height=1.5ex, text depth=0.25ex]
{C&X\\
 D& Y\\ };
\path[->,font=\scriptsize]
(m-1-1) edge node[auto,left] {$g$} (m-2-1)
(m-1-2) edge node[auto,right] {$f'$} (m-2-2)
(m-2-1) edge node[auto,below] {$v$} (m-2-2)
(m-1-1) edge node[auto,above] {$u$} (m-1-2);
\end{tikzpicture}
$$
in which $g\in\L$. Then we have a commutative diagram
$$
\begin{tikzpicture}[description/.style={fill=white,inner sep=2pt}]
\matrix (m) [matrix of math nodes, row sep=3em,
column sep=2.5em, text height=1.5ex, text depth=0.25ex]
{C&A\\
 D& Y\\ };
\path[->,font=\scriptsize]
(m-1-1) edge node[auto,left] {$g$} (m-2-1)
(m-1-2) edge node[auto,right] {$f'\beta$} (m-2-2)
(m-2-1) edge node[auto,below] {$v$} (m-2-2)
(m-1-1) edge node[auto,above] {$\gamma u$} (m-1-2);
\end{tikzpicture}
$$
and so there exists $h:D\to A$ with $hg=\gamma u$ and $f'\beta h=v$. Consequently $\beta h:D\to X$ is such that $(\beta h)g=u$ and $f'(\beta h)=v$ and $f'\in\R$. A similar argument applies to the other case.

Conversely, suppose conditions (1), (2), (3a) and (3b) are satisfied. Then by property (2) we see that $\R\subseteq\L^\Box$ and $\L\subseteq{}^\Box\R$. Suppose that $f:A\to B\in{}^\Box\R$. Notice first that properties (1) and (3) mean that $1_X\in\L\cap\R$ for all $S-$acts $X$. Now by (1) $f=\alpha f'$ with $f'\in\L, \alpha\in\R$. By property (2), $\alpha$ splits with splitting map $\gamma$ such that $f'=\gamma f$ and so by property (3a) $f\in\L$ and $\L={}^\Box\R$. In a similar way, $\R=\L^\Box$.
\end{proof}

Notice that in 3(a) above, $f$ is a retract of $\alpha f$ and that all retracts can be written this way. So this result is simply saying that $\L$ is closed under retracts. Similarly 3(b) is equivalent to $\R$ being closed under retracts.
%
Notice also that by the remarks after Lemma~\ref{box-c-lemma} it follows that if $(\L,\R)$ is a weak factorization system then $\L$ is saturated.

\section{Weak factorization systems and covers of $S-$acts}
In this section we provide a number of examples of weak factorization systems, some of which are related to the existence of covers of $S-$acts, and refer the reader to \cite{bailey-12} and \cite{renshaw-02} for more details of some of the concepts and results used.

\smallskip

We say that a monomorphism $f:X\to Y$ is {\em unitary} if $y\in\im(f)$ whenever $ys\in\im(f)$ and $s\in S$. Clearly this is equivalent to saying that there exists an $S-$act $Z$ such that $Y\cong X\dot\cup Z$ or in other words, $\im(f)$ is a direct summand of $Y$.

\begin{theorem}
Let $S$ be a monoid and let $\U$ be the class of all {\em unitary $S-$monomorphisms} and $\Sp$ be the class of all split $S-$epimorphisms.
Then $(\U,\Sp)$ is a weak factorization system.
\end{theorem}
\begin{proof}
First if $f:X\to Y$ is an $S-$map then define the split epimorphism $\o f:X\dot\cup Y\to Y$ by $\o f(x)=f(x), \o f(y) = y$ for all $x\in X, y\in Y$ and let $\iota:X\to X\dot\cup Y$ be the inclusion. Then it is clear that $f=\o f\iota$ and that $\iota\in\U, \o f\in\Sp$. Now consider the commutative diagram
$$
\begin{tikzpicture}[description/.style={fill=white,inner sep=2pt}]
\matrix (m) [matrix of math nodes, row sep=3em,
column sep=2.5em, text height=1.5ex, text depth=0.25ex]
{A&C\\ A\dot\cup B& D\\ };

\path[->,font=\scriptsize]
(m-1-1) edge node[auto,left] {$f$} (m-2-1)
(m-1-2) edge node[auto,right] {$g$} (m-2-2)
(m-2-1) edge node[auto,below] {$v$} (m-2-2)
(m-1-1) edge node[auto,above] {$u$} (m-1-2);
\end{tikzpicture}
$$
with $f\in\U, g\in\Sp$. There exists $h:D\to C$ with $gh=1_D$ and so we can define an $S-$map $k:A\dot\cup B\to C$ by $k|_A=u, k|_B=hv|_B$ with the required property.

Finally suppose that $f,f',\alpha,\beta$ are as in Proposition~\ref{split-factorization-proposition}. If $f'\beta$ is an epimorphism then clearly so is $f'$. If $g:Y\to A$ is such that $f'\beta g=1_Y$ then $f'\in\Sp$ with splitting morphism $\beta g$. Suppose then that there exists $b\in B, s\in S$ with $bs=f(a)$. Then $\alpha(b)s=\alpha f(a)$ and so $\alpha(b)\in\im(\alpha f)$ as $\alpha f\in\U$. As $\alpha$ is a monomorphism then $b\in\im(f)$ and $f\in\U$ as required.
\end{proof}

Now suppose that $\L$ is a class of $S-$morphisms that contain the unitary $S-$maps and suppose that $(\L,\R)$ is a weak factorization system. Let $g:C\to D\in\R$ and consider the commutative diagram
$$
\begin{tikzpicture}[description/.style={fill=white,inner sep=2pt}]
\matrix (m) [matrix of math nodes, row sep=3em,
column sep=2.5em, text height=1.5ex, text depth=0.25ex]
{C&C\\ C\dot\cup D& D\\ };

\path[->,font=\scriptsize]
(m-1-1) edge node[auto,left] {$f$} (m-2-1)
(m-1-2) edge node[auto,right] {$g$} (m-2-2)
(m-2-1) edge node[auto,below] {$\o g$} (m-2-2)
(m-1-1) edge node[auto,above] {$1_C$} (m-1-2);
\end{tikzpicture}
$$
with $\o g$ as given above. Then there exists $h:C\dot\cup D\to C$ such that $\o g = gh$ and so $g$ is a split epimorphism and it follows that $\R\subseteq \Sp$.

\smallskip

Let $f:A\to B$ be an $S-$map. We say that an $S-$act $P$ is {\em projective with respect to $f$} if for any $S-$map $g:P\to B$ there exists an $S-$map $h:P\to A$ such that $hf=g$.  Then $P$ is called {\em projective} if it is projective with respect to every $S-$epimorphism. For example, it is well-known that free acts are projective.

\begin{lemma}\label{retract-projective-lemma}
Let $S$ be a monoid and suppose that $P$ is projective with respect to the $S-$map $f:A\to B$ and suppose that $Q$ is a retract of $P$. Then $Q$ is projective with respect to $f$.
\end{lemma}
\begin{proof}
Suppose that $\alpha:Q\to P, \beta : P\to Q$ are such that $\beta\alpha=1_Q$ and suppose that $g:Q\to B$ is an $S-$map. Then $g\beta:P\to B$ is an $S-$map and so there exists $h:P\to A$ such that $fh=g\beta$. Then $h\alpha:Q\to A$ is such that $fh\alpha=g\beta\alpha=g$ as required.
\end{proof}

\begin{lemma}\label{retract-lemma}
Let $S$ be a monoid and let $A$ be a retract of the $S-$act $X$.
\begin{enumerate}
\item If $B$ is a retract of $A$ then $B$ is a retract of $X$.
\item If $B$ is a retract of an $S-$act $Y$ then $A\dot\cup B$ is a retract of the $S-$act $X\dot\cup Y$.
\item If $B$ is a direct summand of $A$ then $B$ is a retract of a direct summand of $X$.
\end{enumerate}
\end{lemma}
\begin{proof}
\begin{enumerate}
\item[1 and 2.] These are straightforward.
\item[3.] Let $\alpha : A\to X$ and $\beta :X\to A$ be such that $\beta\alpha=1_A$ and let $Y=\{x\in X|\beta(x)\in B\}$. Then $Y$ is a direct summand of $X$ and $\im(\alpha|_B)\subseteq Y$. Hence $\beta|_Y:Y\to B$ and $B$ is a retract of $Y$.
\end{enumerate}
\end{proof}

Now let $\C$ be a class of $S-$maps and define ${}^\triangle\C$ to be the class of $S-$acts which are projective with respect to each map in $\C$. Notice that by Lemma~\ref{retract-projective-lemma}, ${}^\triangle\C$ is closed under retracts.

\begin{proposition}\label{triangle-split-proposition}
Let $S$ be a monoid and let $\C$ be a class of $S-$epimorphisms closed under pullbacks. Then $D\in{}^\triangle\C$ if and only if every $C\to D\in\C$ splits.
\end{proposition}
\begin{proof}
Suppose that every $C\to D\in\C$ splits and consider the pullback diagram
$$
\begin{tikzpicture}[description/.style={fill=white,inner sep=2pt}]
\matrix (m) [matrix of math nodes, row sep=3em,
column sep=2.5em, text height=1.5ex, text depth=0.25ex]
{P&A\\
 D& B\\};
\path[->,font=\scriptsize]
(m-1-1) edge node[auto,left] {$\o g$} (m-2-1)
(m-1-1) edge node[auto,above] {$u$} (m-1-2)
(m-1-2) edge node[auto,right] {$g$} (m-2-2)
(m-2-1) edge node[auto,below] {$v$} (m-2-2);
\end{tikzpicture}
$$
where $g\in\C$. By assumption, $\o g\in\C$ and so splits with splitting map $f:D\to P$. Define $h:D\to A$ by $h=uf$. Then $gh = guf = v\o gf = v$ and $D\in\C^\triangle$.

Conversely suppose that $D\in\C^\triangle$ and suppose $g:C\to D\in\C$. Then we have a commutative diagram
$$
\begin{tikzpicture}[description/.style={fill=white,inner sep=2pt}]
\matrix (m) [matrix of math nodes, row sep=3em,
column sep=2.5em, text height=1.5ex, text depth=0.25ex]
{&C\\
 D& D\\};
\path[->,font=\scriptsize]
(m-1-2) edge node[auto,right] {$g$} (m-2-2)
(m-2-1) edge node[auto,below] {$1$} (m-2-2);
\end{tikzpicture}
$$
and so there exists $f:D\to C$ such that $gf=1_D$ and $g$ splits.
\end{proof}

Let $\X$ be a class of $S-$acts closed under coproducts, direct summands and retracts. Let $\RX$ denote the class of maps with respect to which each $S-$act in $\X$ is projective and let $\UX$ denote the class of unitary $S-$monomorphisms $f:X\to Y$ such that $Y\setminus\im(f)\in\X$. Note that $\RX$ contains all $S-$isomorphisms and that $\X\subseteq{}^\triangle\RX$. Consider the commutative diagram
$$
\begin{tikzpicture}[description/.style={fill=white,inner sep=2pt}]
\matrix (m) [matrix of math nodes, row sep=3em,
column sep=2.5em, text height=1.5ex, text depth=0.25ex]
{A&&C\\
 A\dot\cup B&&D\\ };
\path[->,font=\scriptsize]
(m-1-1) edge node[auto,above] {$u$} (m-1-3)
(m-2-1) edge node[auto,below] {$v$} (m-2-3)
(m-1-3) edge node[auto,right] {$g$} (m-2-3)
(m-1-1) edge node[auto,left] {$f$} (m-2-1);
\end{tikzpicture}
$$
where $f\in\UX, g\in\RX$. Let $h:B\to C$ be given by the projective property and let $k:A\dot\cup B\to C$ be given by $k|A=u, k|B=h$. Then $f$ has the left lifting property with respect to $g$. Now let $f,f',\alpha,\beta$ be as in Proposition ~\ref{split-factorization-proposition}. Let $P\in\X$ and $v:P\to Y$ be an $S-$map. Then since $f'\beta\in\RX$ there exists an $S-$map $h:P\to A$ such that $f'\beta h=v$ and so $\beta h:P\to X$ is such that $f'(\beta h)=v$ and $f'\in\RX$. If $\alpha f\in\UX$ then $Y=A\dot\cup P$ for some $P\in\X$ and, as we have seen previously, $f\in\U$. Hence $B=A\dot\cup Q$ for some $S-$act $Q$. Let $P'=\{p\in P|\beta(p)\in Q\}$ and note that $\im(\alpha|_Q)\subseteq P$ and $P'$ is a direct summand of $P$ and so $P'\in\X$. Consequently $\alpha|_Q:Q\to P'$ splits and so $Q$ is a retract of $P'$ and so $Q\in\X$. This means that $f\in\UX$.

\medskip

Suppose now that every $S-$act has an $\X-$precover. Let $f:A\to B$ be an $S-$map and let $X\to B$ be an $\X-$precover for $B$. Then the factorization $A\to A\dot\cup X\to B$ shows that $(\UX,\RX)$ is a weak factorization system.

\medskip

Conversely, suppose that $(\UX,\RX)$ as defined above is a weak factorization system and let $X$ be an $S-$act. Suppose that there exists $A\in\X$ and an $S-$map $h:A\to X$. Then $h$ factorizes as $gf$ where $f:A\to A\dot\cup Y\in\UX$ and $g:A\dot\cup Y\to X\in\RX$. Notice then that $Y\in\X$. Let $B\in\X$ be an $S-$act and suppose that $u:B\to X$ is an $S-$map. Consider the commutative diagram
$$
\begin{tikzpicture}[description/.style={fill=white,inner sep=2pt}]
\matrix (m) [matrix of math nodes, row sep=3em,
column sep=2.5em, text height=1.5ex, text depth=0.25ex]
{A&&A\dot\cup Y\\
 A\dot\cup B&&X\\ };
\path[->,font=\scriptsize]
(m-1-1) edge node[auto,above] {$f$} (m-1-3)
(m-2-1) edge node[auto,below] {$v$} (m-2-3)
(m-1-3) edge node[auto,right] {$g$} (m-2-3)
(m-1-1) edge node[auto,left] {$\iota$} (m-2-1);
\end{tikzpicture}
$$
where $v|_A=h$ and $v|_B=u$. By the lifting property there exists $k:A\dot\cup B\to A\dot\cup Y$ such that $gk=v$. Then $gk|_B=u$ and $g:A\dot\cup Y\to X$ is an $\X-$precover.

Consequently we have shown

\begin{theorem}
Let $S$ be a monoid and let $\X$ be a class of $S-$acts closed under coproducts, direct summands and retracts and define $\UX$ and $\RX$ as above. Then every $S-$act has an $\X-$precover if and only if $(\UX,\RX)$ is a weak factorization system and for all $S-$acts $X$ there exists $Y\in\X$ such that $\hom(Y,X)\ne\emptyset$.
\end{theorem}

\medskip

Now let $\Pr$ be the class of projective $S-$acts and let $\E$ denote the class of all $S-$epimorphisms. By definition, every projective $S-$act is projective with respect to every $S-$epimorphism. It is also well known that $\Pr$ is closed with respect to coproducts and direct summands. That retracts of projective $S-$acts are projective follows from Lemma~\ref{retract-projective-lemma}.
%
%
From~\cite[Proposition 5.8]{bailey-12} we see that every $S-$act has a $\Pr-$precover and consequently we deduce
\begin{corollary}
$(\UPr,\E)$ is a weak factorization system.
\end{corollary}
%
%
%

Since retracts of free $S-$acts are not necessarily free then it is clear that we cannot necessarily replace projective by free in the above proposition. Also, if $\Fr$ denotes the class of free $S-$acts then $\Fr\subsetneq {}^\triangle\R_\Fr = \Pr$. It also follows by a similar argument to that outlined previously that if $(\L,\R)$ is a weak factorization system in which $\UPr\subseteq\L$ then $\R\subseteq\E$.

\medskip


Let $\psi:X\to Y$ be an $S-$epimorphism. We say that $\psi$ is {\em a pure epimorphism} if for every finitely presented $S-$act $M$ and every $S-$map $f:M\to Y$ there exists $g:M\to X$ such that $f=\psi g$. For more details of pure epimorphisms and their connection with covers of acts see~\cite{bailey-12}.
Now Let $\FP$ be the set of finitely presented acts, $\IFP$ the class of acts whose indecomposable components are in $\FP$ (in other words all coproducts of acts in $\FP$) and let $\RIFP$ be the class of all retracts of acts in $\IFP$.

Let $\UFP$  be the class of unitary monomorphisms $f:X \to Y$, where $Y\setminus\im(f)\in\RIFP$, and let $\PE$ be the class of pure epimorphisms. From~\cite{bailey-12} we easily see that acts in $\IFP$ are projective with respect to pure epimorphisms and therefore by Lemma~\ref{retract-projective-lemma} so are acts in $\RIFP$.
%
Therefore $\R_\IFP=\PE$.
From Lemma~\ref{retract-lemma} we can easily deduce that $\RIFP$ is closed under retracts, coproducts and direct summands.

We use \cite[Corollary 4.14 and Proposition 4.3]{bailey-12} to deduce that every act has an epimorphic $\IFP-$precover. Specifically, note that $\IFP$ is closed under coproducts and direct summands, and the generator $S \in \IFP$ and so for every $S-$act $A$, $\hom(S,A) \neq \emptyset$. Finally every indecomposable $\IFP-$act (that is, an $\FP-$act) is bounded in size by $\max\{\aleph_0,|S|\}$. 
Consequently every act has an epimorphic $\IFP-$precover and so an $\RIFP-$precover.

\smallskip

We can therefore deduce
\begin{corollary}
$(\UFP,\PE)$ is a weak factorization system.
\end{corollary}

\medskip

We may wish to ask, for which classes of monomorphism $\L$ is it true that $\R\subseteq\E$? We can supply a partial answer in the case of monoids with a left zero by observing the following connection with the concept of injectivity.

Recall that we say that an $S-$act $X$ is {\em injective} if for all $S-$monomorphisms $f:A\to B$ and all maps $g:A\to X$ there exists an $S-$map $h:B\to X$ such that $hf=g$. It is well known (see~\cite{ahsan-08} or~\cite{kilp-00}) that $X$ is injective if and only if every monomorphism with $X$ as domain, splits. Now let $\C$ be a class of maps and define $\C^\triangle$ to be the class of $S-$acts {\em injective to $\C$}. In other words
$$
\C^\triangle = \{C\in \text{Act}-S: C\to1\in\C^\Box\}
$$
where $1$ is the 1-element $S-$act.
The proof of the following result is similar to that of Proposition~\ref{triangle-split-proposition}.
\begin{proposition}
Let $S$ be a monoid and let $\C$ be a class of $S-$monomorphisms closed under pushouts. Then $C\in\C^\triangle$ if and only if every $C\to D\in\C$ splits.
\end{proposition}
%

If $S$ is a monoid with a left zero element $z$, then every $S-$act $A$ contains at least 1 fixed point since for all $a\in A, s\in S,  (az)s=az$. Denote the set of fixed points of an act $A$ by $\Fix(A)$. Given an $S-$map $g:C\to D$ and $d\in \Fix(D)$ let $K_d = \{c\in C|g(c)=d\}$. Then $K_d$ is either empty or an $S-$subact of $C$.

\begin{lemma}\label{r-epi-lemma}
Let $S$ be a monoid with a left zero $z$ and let $\C$ be a class of $S-$morphisms containing the inclusion $\{z\}\to S$.
If $g:C\to D\in\C^\Box$ then $K_d\in\C^\triangle$ for all $d\in \Fix(D)$ with $K_d\ne\emptyset$. Moreover $g$ is an epimorphism if and only if $\Fix(D)\subseteq\im(g)$.
\end{lemma}

\begin{proof}
Let $g:C\to D\in\C^\Box$ and let $d\in\Fix(D)$ with $\iota:K_d\to C$ the inclusion. Let $1$ be the 1-element $S-$act and consider the commutative diagram
$$
\begin{tikzpicture}[description/.style={fill=white,inner sep=2pt}]
\matrix (m) [matrix of math nodes, row sep=3em,
column sep=2.5em, text height=1.5ex, text depth=0.25ex]
{A&K_d\\
&C\\
 & D\\
B&1\\};
\path[->,font=\scriptsize]
(m-1-1) edge node[auto,left] {$f$} (m-4-1)
(m-1-1) edge node[auto,above] {$u$} (m-1-2)
(m-1-2) edge node[auto,right] {$\iota$} (m-2-2)
(m-2-2) edge node[auto,right] {$g$} (m-3-2)
(m-3-2) edge node[auto,right] {$$} (m-4-2)
(m-4-1) edge node[auto,above] {$0_d$} (m-3-2)
(m-4-1) edge node[auto,below] {$$} (m-4-2);
\end{tikzpicture}
$$
where $f\in\C$ and where $0_d(b)=d$ for all $b\in B$. Since $g\in\C^\Box$ then there exists $h:B\to C$ with $gh=0_d$ and $hf=\iota u$ and so since $gh=0_d$ there exists $h':B\to K_d$ with $\iota h'=h$. Consequently $\iota h' f = hf = \iota u$ and so $h'f=u$ and $K_d\in\C^\triangle$.

\smallskip

Now suppose that $\Fix(D)\subseteq\im(g)$ and let $d\in D$ and define $p:S\to D$ by $p(s)=ds$. Let $c\in C$ be such that $g(c)=dz$ and define $\{z\}\to C$ by $z\mapsto c$. Since $\{z\}\to S\in\C$ then there exists $h:S\to C$ such that $gh=p$. Hence $d=gh(1)$ and so $g$ is an epimorphism.
\end{proof}

\smallskip

If $\X$ is a class of $S-$acts then we shall denote by $\X-mono$ the class of $S-$monomorphisms $f:X\to Y$ such that the Rees quotient $Y/X\in\X$.
We now deduce from Lemma~\ref{r-epi-lemma}

\begin{proposition}
Let $S$ be a monoid with a left zero and suppose that $\X$ is a class of $S-$acts that contain $S$. If $(\L,\R)$ is a weak factorizations system such that $\X-mono\subseteq\L$ then $\R\subseteq\E$ if and only if for every $g:C\to D\in\R, \Fix(D)\subseteq\im(g)$.
\end{proposition}

Let $\F$ denote the class of flat $S-$acts. We cannot at this stage determine whether or not there is a class $\R$ such that ($\F-mono,\R)$ is a weak factorization system. However we do have

\begin{theorem}
$\F-mono$ is saturated.
\end{theorem}
\begin{proof}
We prove this in three steps
\begin{enumerate}
\item Consider the pushout diagram
$$
\begin{tikzpicture}[description/.style={fill=white,inner sep=2pt}]
\matrix (m) [matrix of math nodes, row sep=3em,
column sep=2.5em, text height=1.5ex, text depth=0.25ex]
{A&C\\
 B&P\\ };
\path[->,font=\scriptsize]
(m-1-1) edge node[auto,above] {$u$} (m-1-2)
(m-2-1) edge node[auto,below] {$v$} (m-2-2)
(m-1-2) edge node[auto,right] {$g$} (m-2-2)
(m-1-1) edge node[auto,left] {$f$} (m-2-1);
\end{tikzpicture}
$$
where $f\in\F-$mono.
Since pushouts of monomorphisms are monomorphisms then we can
define $h:B/A\to P/C$ by $h(\o b) = \o{v(b)}$. Then $h$ is a well-defined $S-$map since if $\o b=\o{b'}$ then either $b=b'$, in which case $\o{v(b)}=\o{v(b')}$, or $b=f(a), b'=f(a')$ for some $a,a'\in A$. In this case, $\o{v(b)}=\o{gu(a)}=\o{gu(a')}=\o{v(b')}$ as required.

In a similar way, if $\o{v(b)}=\o{v(b')}$ then either $v(b)=v(b')$ or $v(b)=g(c), v(b')=g(c')$ for some $c,c'\in C$. In the latter case, since $P$ is a pushout there exists $a,a'\in A$ such that $b=f(a), b'=f(a')$ and so $\o b=\o{b'}$ in $B/A$. In the former case, from the comments on pushouts at the end of Section 1 it follows that either $b=b'$, in which case $\o b=\o{b'}$, or $b=f(a), b'=f(a')$ for some $a,a'\in A$. In this case, $\o b=\o{f(a)}=\o{f(a')}=\o{b'}$. Consequently we see that $h$ is a monomorphism.

Finally if $p\in P$ then either $p=v(b)$ or $p=g(c)$ for some $b\in B, c\in C$. In the first case $\o p = \o{v(b)} = h(\o b)$. In the second case, for any $a\in A$ we have $\o p = \o{g(c)} = \o{gu(a)} = \o{vf(a)} = h(\o{f(a)})$ and $h$ is an isomorphism.

It therefore follows that $g\in\F-$mono.
\item It follows immediately from Lemmas~\ref{rees-flat-lemma1} and~\ref{rees-stable-lemma1} that the composite of two maps in $\F-$mono is in $\F-$mono. Now suppose that $\lambda$ is a limit ordinal and that for all $i\le j<\lambda$, $f_{ij}:A_i\to A_j$ are such that $f_{ij}\in\F-$mono and that if $j$ is a limit ordinal then $f_{ij}$ is a colimit. Let $(A_\lambda,f_{i\lambda})$ be the directed colimit and note from~\cite[Corollary 3.6]{renshaw-86a} that $f_{i\lambda}$ is a monomorphism. Let $\o{f_{ij}}:A_i/A_0\to A_j/A_0$ be given by $\o{f_{ij}}(\o{a_i}) = \o{f_{ij}(a_i)}$. By~\cite[Lemma 3.6]{renshaw-86} we can deduce that $(A_\lambda/A_0,\o{f_{i\lambda}})$ is the direct limit of $(A_i/A_0,\o{f_{ij}})$ and so by \cite[Corollary 3.5]{renshaw-86} we deduce that $A_\lambda/A_0$ is flat as required.

\item Suppose that If $f:A\to B\in\F-$mono and $g:A\to C$ are such that there exist maps $\alpha:C\to B$ and $\beta:B\to C$ with $\beta\alpha=1_C, \alpha g=f$ and $\beta f=g$. It is clear that retracts of monomorphisms are monomorphisms. Consider the following commutative diagram
$$
\begin{tikzpicture}[description/.style={fill=white,inner sep=2pt}]
\matrix (m) [matrix of math nodes, row sep=3em,
column sep=2.5em, text height=1.5ex, text depth=0.25ex]
{&A\\
 B& C\\
B/A&C/A\\};
\path[->,font=\scriptsize]
(m-1-2) edge node[auto,left] {$f$} (m-2-1)
(m-1-2) edge node[auto,right] {$g$} (m-2-2)
(m-2-1.340) edge[<-] node[auto,below] {$\alpha$} (m-2-2.200)
(m-2-1.20) edge node[auto,above] {$\beta$} (m-2-2.160)
(m-2-1) edge node[auto,right] {$$} (m-3-1)
(m-2-2) edge node[auto,right] {$$} (m-3-2)
(m-3-1.340) edge[<-] node[auto,below] {$\o\alpha$} (m-3-2.200)
(m-3-1.20) edge node[auto,above] {$\o\beta$} (m-3-2.160);
\end{tikzpicture}
$$
where $\o\alpha(\o c) = \o{\alpha(c)}$ and $\o\beta(\o b) = \o{\beta(b)}$. It is easy to see that $\o\alpha$ and $\o\beta$ are well defined $S-$maps and that $\o\beta\o\alpha=1_{C/A}$. Hence from Lemma~\ref{flat-retract-lemma1} we see that $C/A$ is flat and so $g\in\F-$mono.
\end{enumerate}
\end{proof}

\section{Cofibrantly generated systems}

Finding classes of acts that form a weak factorization system seems to be quite a difficult task. However, in certain cases we can `generate' such a system from a given class of $S-$maps. In particular, if $(\L,\R)$ is a weak factorization system such that there exists a {\bf set} $\C$ of $S-$maps with $\L=\cof(\C)$ then we say that the system is {\em cofibrantly generated} by $\C$.

Let $\C$ be a class of morphisms and let $\L={}^\Box\C$ and $\R=\L^\Box$. Then it is easy to see from Lemma~\ref{box-algebra-lemma}(3) that the pair $(\L,\R)$ satisfy $\L={}^\Box\R$ and $\R=\L^\Box$. 

\begin{theorem}[{Cf. \cite[Proposition 1.3]{beke-00}}]\label{cofibration-theorem}
Let $S$ be a monoid and let $\C$ be a {\rm\bf set} of $S-$morphisms closed under coproducts and let $\R=\C^\Box, \L={}^\Box\R$. Then $(\L,\R)$ is a weak factorization system.
\end{theorem}
\begin{proof}
That the pair $(\L,\R)$ satisfy $\L={}^\Box\R$ and $\R=\L^\Box$ follows from Lemma~\ref{box-algebra-lemma}(3). Notice also that $\C\subseteq\L$.

The following is a modified version of that found in~\cite{beke-00} but also contains ideas found in~\cite{hovey-88}. Let $g:X\to Y$ be a morphism and consider the following construction $(P(g),\theta(g),\phi(g))$. Let $\S$ be the set of all commutative squares
$$
\begin{tikzpicture}[description/.style={fill=white,inner sep=2pt}]
\matrix (m) [matrix of math nodes, row sep=3em,
column sep=2.5em, text height=1.5ex, text depth=0.25ex]
{A&X\\
 B&Y\\ };
\path[->,font=\scriptsize]
(m-1-1) edge node[auto,above] {$u$} (m-1-2)
(m-2-1) edge node[auto,below] {$v$} (m-2-2)
(m-1-2) edge node[auto,right] {$g$} (m-2-2)
(m-1-1) edge node[auto,left] {$f$} (m-2-1);
\end{tikzpicture}
$$
with $f\in\C$. If $\S=\emptyset$ then we have a factorisation of $g$ as $X\stackrel{1_X}{\to}X\stackrel{g}{\to}Y$ and by Lemma~\ref{box-c-lemma}, $1_X\in\L$. It easily follows in that case that $g\in\R$ as required.

Assume now that $\S\ne\emptyset$. Let $\bar f:A_{\S}\to B_{\S}$ be the coproduct of all these maps, let $A_{\S}\to X$ be the natural map induced by the coproduct and consider the pushout diagram
$$
\begin{tikzpicture}[description/.style={fill=white,inner sep=2pt}]
\matrix (m) [matrix of math nodes, row sep=3em,
column sep=2.5em, text height=1.5ex, text depth=0.25ex]
{A_{\S}&X\\
 B_{\S}&P(g)\\ };
\path[->,font=\scriptsize]
(m-1-1) edge node[auto,above] {$$} (m-1-2)
(m-2-1) edge node[auto,below] {$$} (m-2-2)
(m-1-2) edge node[auto,right] {$\theta(g)$} (m-2-2)
(m-1-1) edge node[auto,left] {$\bar f$} (m-2-1);
\end{tikzpicture}
$$
By the pushout property, there exists a unique map $\phi(g):P(g)\to Y$ such that the diagram
$$
\begin{tikzpicture}[description/.style={fill=white,inner sep=2pt}]
\matrix (m) [matrix of math nodes, row sep=3em,
column sep=2.5em, text height=1.5ex, text depth=0.25ex]
{A_{\S}&X&\\
 B_{\S}&P(g)&\\ 
&&Y\\};
\path[->,font=\scriptsize]
(m-1-2) edge node[auto,right] {$g$} (m-3-3)
(m-2-1) edge node[auto,left] {$$} (m-3-3)
(m-2-2) edge node[auto,left] {$\phi(g)$} (m-3-3)
(m-1-1) edge node[auto,above] {$$} (m-1-2)
(m-2-1) edge node[auto,below] {$$} (m-2-2)
(m-1-2) edge node[auto,left] {$\theta(g)$} (m-2-2)
(m-1-1) edge node[auto,left] {$\bar f$} (m-2-1);
\end{tikzpicture}
$$
commutes. So $g$ can be factorized as $g=\phi(g)\theta(g)$. For notational convenience let $P_0(g)=P(g), \phi_0(g)=\phi(g), \theta_0(g)=\theta(g)$.

By way of transfinite induction suppose that for all ordinals $\alpha < \epsilon$ we have an $S-$act $P_\alpha(g)$ and $S-$maps $\theta_\alpha(g):X\to P_\alpha(g)$ and $\phi_\alpha(g):P_\alpha(g)\to Y$ such that $g=\phi_\alpha(g)\theta_\alpha(g)$ and suppose also that for all ordinals $\beta<\gamma<\delta<\epsilon$ there are $S-$maps $\psi_{\beta,\gamma}:P_\beta(g)\to P_\gamma(g), \psi_{\gamma,\delta}:P_\gamma(g)\to P_\delta(g), \psi_{\beta,\delta}:P_\beta(g)\to P_\delta(g)$ such that $\psi_{\gamma,\delta}\psi_{\beta,\gamma} = \psi_{\beta,\delta}$. If $\epsilon=\alpha+1$ then we let $P_{\alpha+1}(g) = P(\phi_\alpha(g)), \theta_{\alpha+1}(g) = \theta(\phi_\alpha(g))\theta_\alpha(g)$ and $\phi_{\alpha+1}(g)=\phi(\phi_\alpha(g))$ and we define $\psi_{\beta,\alpha+1} = \theta_{\alpha+1}(g)\psi_{\beta,\alpha}$ for all $\beta\le\alpha$.

If $\epsilon$ is a limit ordinal we construct $(P_\epsilon(g),\theta_\epsilon(g),\phi_\epsilon(g))$ by directed colimits. Specifically we have a directed system of $S-$acts $(P_\alpha(g),\psi_{\alpha,\beta})$ with directed colimit $(P_\epsilon(g),\psi_{\alpha,\epsilon})$
$$
\begin{tikzpicture}[description/.style={fill=white,inner sep=2pt}]
\matrix (m) [matrix of math nodes, row sep=3em,
column sep=2.5em, text height=1.5ex, text depth=0.25ex]
{P_\alpha(g)&&&&P_{\beta}(g)\\
&&P_\epsilon(g)&&\\ 
&&Y&&\\};
\path[->,font=\scriptsize]
(m-1-1) edge node[auto,left] {$\phi_\alpha(g)$} (m-3-3)
(m-1-5) edge node[auto,right] {$\phi_{\beta}(g)$} (m-3-3)
(m-2-3) edge node[auto,above left] {$\phi_\epsilon(g)$} (m-3-3)
(m-1-1) edge node[auto,above] {$\psi_{\alpha,\beta}$} (m-1-5)
(m-1-1) edge node[auto,right] {$\psi_{\alpha,\epsilon}$} (m-2-3)
(m-1-5) edge node[auto,left] {$\psi_{\beta,\epsilon}$} (m-2-3);
\end{tikzpicture}
$$
Let $\theta_\epsilon(g) = \psi_{0,\epsilon}\theta_0(g):X\to P_\epsilon(g)$ and $\phi_\epsilon(g) : P_\epsilon(g)\to Y$ be as above.


Let $\gamma$ be a cardinal such that $\gamma =\sup\{|\dom(f)||S|: f\in\C\}$  and let $\kappa$ be a $\gamma-$filtered ordinal.

Now
consider the factorisation $g = \phi_\kappa(g)\theta_\kappa(g) : X\to P_\kappa(g)\to Y$ with $\kappa$ as above. From Lemma~\ref{cof-lemma} we see that $\theta_\kappa(g)\in\cof(\C)=\L$. Suppose now that
$$
\begin{tikzpicture}[description/.style={fill=white,inner sep=2pt}]
\matrix (m) [matrix of math nodes, row sep=3em,
column sep=2.5em, text height=1.5ex, text depth=0.25ex]
{A&P_\kappa(g)\\
 B&Y\\ };
\path[->,font=\scriptsize]
(m-1-1) edge node[auto,above] {$u$} (m-1-2)
(m-2-1) edge node[auto,below] {$v$} (m-2-2)
(m-1-2) edge node[auto,right] {$\phi_\kappa(g)$} (m-2-2)
(m-1-1) edge node[auto,left] {$f$} (m-2-1);
\end{tikzpicture}
$$
is a commutative diagram with $f\in\C$.

By Lemma~\ref{small-lemma}, there exists $\delta<\kappa$ and $u':A\to P_\delta(g)$ such that $u=\psi_{\delta,\kappa}u'$. But then by construction there exists $v':B\to P_{\delta+1}(g)$ and
$$
\begin{tikzpicture}[description/.style={fill=white,inner sep=2pt}]
\matrix (m) [matrix of math nodes, row sep=3em,
column sep=2.5em, text height=1.5ex, text depth=0.25ex]
{A&&P_\kappa(g)\\
 B&&Y\\ };
\path[->,font=\scriptsize]
(m-2-1) edge node[auto,above] {$\psi_{\delta+1,\kappa}v'\ \ \ \ $} (m-1-3)
(m-1-1) edge node[auto,above] {$u$} (m-1-3)
(m-2-1) edge node[auto,below] {$v$} (m-2-3)
(m-1-3) edge node[auto,right] {$\phi_\kappa(g)$} (m-2-3)
(m-1-1) edge node[auto,left] {$f$} (m-2-1);
\end{tikzpicture}
$$
the required `lifting' diagram so that $\phi_\kappa(g)\in\R$.
\end{proof}

Notice that in the preceding proposition, $\L=\cof(\C)$ and $\R=\left(\cof(\C)\right)^\Box$.

\begin{corollary}
Let $S$ be a monoid and let $\C$ be a {\rm\bf set} of $S-$morphisms closed under coproducts. Then $\cof(\C)=\ret(\C)$ and $(\ret(\C),\C^\Box)$ is a weak factorization system cofibrantly generated by $\C$.
\end{corollary}
\begin{proof}
One way round follows immediately from Lemma~\ref{cof-lemma}. Suppose then that $f:A\to B\in\cof(\C)$. By (the proof of) Proposition~\ref{cofibration-theorem} it follows that $f=gf'$ with $f'$ a transfinite composition of pushouts of maps in $\C$ and $g:C\to B\in\left(\cof(\C)\right)^\Box$. We then have a diagram
$$
\begin{tikzpicture}[description/.style={fill=white,inner sep=2pt}]
\matrix (m) [matrix of math nodes, row sep=3em,
column sep=2.5em, text height=1.5ex, text depth=0.25ex]
{A&&C\\
 B&&B\\ };
\path[->,font=\scriptsize]
(m-2-1) edge node[auto,above] {$h$} (m-1-3)
(m-1-1) edge node[auto,above] {$f'$} (m-1-3)
(m-2-1) edge node[auto,below] {$1_B$} (m-2-3)
(m-1-3) edge node[auto,right] {$g$} (m-2-3)
(m-1-1) edge node[auto,left] {$f$} (m-2-1);
\end{tikzpicture}
$$
and so $f$ is a retract of $f'$ as required.
\end{proof}

\bigskip

\section{Centred $S-$acts}

Suppose that $S$ is now a monoid with a zero and that all acts are {\em centred}, that is to say they contain a unique fixed point. Let the 1-element $S-$act here be denoted by $0$.

Suppose in addition that $\X$ is a given class of centred $S-$acts and $(\L,\R)$ is a weak factorization system for the category of centred $S-$acts with the property that $0\to X\in\L$ if and only if $X\in\X$. Then given a centred $S-$act $A$ the unique map $0\to A$ factorises as $0\to A^\ast\to A$ with $0\to A^\ast \in\L$ and $A^\ast\to A\in\R$. It follows by our assumptions that  $A^\ast\in\X$.

Suppose then that $X\in\X$ is a centred right $S-$act and suppose that $X\to A$ is an $S-$map. Then we can construct the commutative diagram
$$
\begin{tikzpicture}[description/.style={fill=white,inner sep=2pt}]
\matrix (m) [matrix of math nodes, row sep=3em,
column sep=2.5em, text height=1.5ex, text depth=0.25ex]
{0&&A^\ast\\
 X&&A\\ };
\path[->,font=\scriptsize]
(m-2-1) edge node[auto,above] {$$} (m-1-3)
(m-1-1) edge node[auto,above] {$$} (m-1-3)
(m-2-1) edge node[auto,below] {$$} (m-2-3)
(m-1-3) edge node[auto,right] {$$} (m-2-3)
(m-1-1) edge node[auto,left] {$$} (m-2-1);
\end{tikzpicture}
$$
where the diagonal map comes from the weak factorization system. This demonstrates that every centred $S-$act has an $\X-$precover in the class of centred $S-$acts.

\medskip

It is easy to demonstrate that $0\to X\in\F-mono$ if and only if $X$ is flat and that $\UPr\subseteq\F-mono$ and so $\fib(\F-mono)\subseteq\E$. However so far we have been unable to show that $\F-mono$ is not only saturated but also cofibrantly generated.

\smallskip

In the case of modules over a unitary ring, it can indeed be shown that $\F-mono$ is cofibrantly generated but the proof seems to depend on the additive structure of the category of modules. We would however like to conjecture that this result is also true for the category of centred acts over a monoid with zero. Little however has been written on the homological aspects of the category of centred $S-$acts.


\begin{thebibliography}{00}
\bibitem{adamek-02} J. Ad\'amek, H. Herrlich, J. Rosick\'y and W Tholen, On a generalised small-object argument for the injective subcategory problem, {\it Cahiers de Topologie et Geometrie Differentielle Categoriques}, XLIII-2 (2002), 83--106.

\bibitem{ahsan-08} Javed Ahsan and Liu Zhongkui, {\it A Homological Approach to the Theory of Monoids}, Science Press, Beijing, (2008).

\bibitem{aldrich-01} Aldrich, S.T., Enochs, E.E., Garcia Rozas, J.R. and L. Oyonarte, Covers and Envelopes in Grothendiek categories. Flat covers of complexes with applications, {\em J. Algebra} {\bf 243} (2001), 615--630.

\bibitem{bailey-12} Bailey, Alex and James Renshaw, Covers of acts over monoids and pure epimorphisms, to appear in {\em Proc. Edinburgh Math. Soc.}

\bibitem{bailey-13} Bailey, Alex and James Renshaw, Covers of acts II, to appear in {\em Semigroup Forum}.

\bibitem{beke-00} Tibor Beke, Sheafifiable homotopy model categories, {\it Math. Proc. Camb. Phil. Soc.} (2000),129,447-475.

\bibitem{hovey-88} M. Hovey, {\it Model Categories}, Mathematical Surveys and Monographs, no. 63 (American Mathematical Society, 1988).

\bibitem{howie-95} J.M. Howie, {\it Fundamentals of Semigroup Theory}, London Mathematical Society Monographs, (OUP, 1995).

\bibitem{jeck-06} Thomas Jeck, {\it Set Theory}, Third Millenium Edition, Springer Monographs in Mathematics, Springer (2006).

\bibitem{kilp-00}Kilp, Mati, Knauer, Ulrich and Alexander V. Mikhalev, {\it Monoids, Acts and Categories}, De Gruyter Expositions in Mathematics {\bf (29)}, (Walter de Gruyter, Berlin, New York, 2000).

\bibitem{renshaw-86a} James Renshaw, Extension and Amalgamation in Monoids and semigroups, {\it Proc. London Math. Soc. (3), 52 (1986), 119--141.}

\bibitem{renshaw-86} James Renshaw, Flatness and Amalgamation in Monoids and semigroups, {\it J. London Math. Soc. (2), 33 (1986), 73--88.}

\bibitem{renshaw-91} James Renshaw, Subsemigroups of free products of semigroups, {\it Proc. Edin. Math. Soc}, {\bf 34} (1991) 337--357.

\bibitem{renshaw-02} James Renshaw, Stability and Flatness in Acts over Monoids, {\it Colloquium Mathematicum}, {\bf 92}(2), (2002) 267--293.



\bibitem{rosicky-02} J. Rosick\'y, Flat covers and factorizations, {\it Journal of Algebra} 253 (2002) 1--13.

\end{thebibliography}
\end{document}